\documentclass[11pt]{article}
\usepackage[utf8]{inputenc}
\usepackage[T1]{fontenc}
\usepackage{amsthm, amsmath, amssymb, mathrsfs, mathtools}
\usepackage[all]{xy}
\usepackage{fancyhdr}
\usepackage{textcomp}
\usepackage{tabto}
\usepackage{biblatex}
\author{Skyler Marks}
	\newcommand{\xto}{\xrightarrow}

	\newcommand{\pd}[2][]{\frac{\partial #1}{\partial #2}}
	
	\DeclareMathOperator{\id}{Id}

	\DeclareMathOperator{\ord}{ord}
	\DeclareMathOperator{\opens}{\textbf{opens}}

	\DeclareMathOperator{\Ker}{Ker}
	\DeclareMathOperator{\im}{im}
	\DeclareMathOperator{\ima}{Im}

	\makeatletter
\def\tagform@#1{\maketag@@@{\ignorespaces(#1\unskip\@@italiccorr)}}
\usepackage[letterpaper, top=3cm, inner=3.3cm, outer=3.3cm]{geometry}
\makeatother
\newtheorem{theorem}{Theorem}[section]
\newtheorem{corollary}{Corollary}[theorem]
\newtheorem{lemma}[theorem]{Lemma}

\newtheorem{definition}{Definition}[section]

\theoremstyle{remark}
\newtheorem{remark}{Remark}[section]

\nonstopmode

\title{The Kodaira Embedding Theorem}
\author{Skyler Marks\thanks{Boston University}}

\begin{document}
\maketitle

\begin{abstract}
	Chow's Theorem and GAGA are renowned results demonstrating the algebraic
	nature of projective manifolds and, more broadly, projective analytic
	varieties. However, determining if a particular manifold is projective is
	not, generally, a simple task. The Kodaira Embedding Theorem provides an
	intrinsic characterization of projective varieties in terms of line bundles;
	in particular, it states that a manifold is projective if and only if it
	admits a positive line bundle. We prove only the `if' implication in this
	paper, giving a sufficient condition for a manifold bundle to be embedded in
	projective space. 

	Along the way, we prove several other interesting results. Of particular
	note is the Kodaira-Nakano
	Vanishing Theorem, a crucial tool for eliminating higher cohomology of
	complex	manifolds, as well as Lemmas \ref{pos} and \ref{pullback}, which provide
	important relationships between divisors, line bundles, and blowups. 

	Although this treatment is relatively self-contained, we omit a rigorous
	development of Hodge theory, some basic complex analysis results, and some theorems regarding Čech cohomology
	(including Leray's Theorem).

	Much of the development follows \autocite{GrHa}, supplementing with
	\autocite{Huy} regularly and drawing on other sources as needed.
\end{abstract}

\pagebreak
\tableofcontents
\section{Preliminaries}
\subsection{Complex Manifolds}
\begin{definition} 
	A \textbf{complex structure} on a manifold is a system of charts on $M$ to
	open balls in $\mathbb C^n$. 
\end{definition}
\begin{definition}
	The \textbf{complexified tangent bundle} $T_{\mathbb C}M$ of a manifold $M$
	is the set of linear functions $f:C^{\infty}(M,\mathbb C)\to
	C^\infty(M,\mathbb C)$ satisfying
	$f(g\cdot h) = hf(g)+gf(h)$. These functions are called
	\textbf{derivations}.
\end{definition}
\begin{definition}
	The \textbf{holomorphic part} $T_{\partial}$ of the complexified tangent space is the set
	of all derivations in $T_\mathbb CM$ which vanish on antiholomorphic
	functions (that is, functions $f$ with $\bar f$ holomorphic). The
	\textbf{antiholomorphic part} $T_{\bar\partial}$ of the complexified tangent space is the set
	of all derivations which vanish on holomorphic functions. 
\end{definition}
\begin{lemma}
	The complex structure on a complex manifold $M$ induces a decomposition on the complexified tangent
	bundle of
	the manifold $M$:
	\[T_\mathbb CM = T_\partial M \oplus T_{\bar\partial}M\]
\end{lemma}
\begin{proof}
	The local holomorphic coordinates $z_1,...,z_n$ induce a local frame for the
	tangent bundle:
	\[\left\{\pd{z_1}, ..., \pd{z_n}, \pd{\bar{z}_1}, ...,
	\pd{\bar{z}_n}\right\}\] 
	Of holomorphic and antiholomorphic derivations. Furthermore, note that since
	the transition functions are holomorphic, the decomposition remains well
	defined from chart to chart (even though the basis may change). 
\end{proof}
\begin{remark}
	This decomposition induces decompositions on the cotangent bundle and on
	each wedge product of the cotangent bundle, of the form 
	\[\bigwedge^k T^*M = \bigoplus_{p+q=k}
	\left[\bigwedge^pT^*_{\partial}M\right]\wedge\left[\bigwedge^qT^*_{\bar\partial}M\right]\]
	We let
	$\Omega^{(p,
	q)}(M)\coloneq\left[\bigwedge^pT^*_{\partial}M\right]\wedge\left[\bigwedge^qT^*_{\bar\partial}M\right]$,
	and call elements of $\Omega^{(p, q)}(M)$ \textbf{differential forms of type
	$(p, q)$}. Restricting the exterior differential to $\Omega^{p,q}$ yields a
	map who's image is contained in $\Omega^{(p+1, q)}\oplus\Omega^{(p, q+1)}$;
	projecting onto each of these factors yields operators which we call
	\textbf{Dolbeault operators}
	$\partial$ and $\bar\partial$. Note that $d = \partial + \bar\partial$
\end{remark}
\begin{remark}
	Complex conjugation acts naturally on the wedge products of the complexified cotangent
	bundle by the rule:
	\[\overline{\sum_{I, J}f_{I, J}dz_I\wedge dz_J} = \sum_{I, J}\bar f_{I,
	J}d\bar z_I\wedge d\bar z_J\]
	moreover, $\overline{\Omega^{(p, w)}} = \Omega^{(q, p)}$
\end{remark}
\begin{lemma}
	The above construction creates a double complex (see figure \ref{dolbcom}).
	\begin{figure}[h]\label{dolbcom}
	\label{fig}
	\centering
	\bigskip
	\centerline{
	\xymatrix{
			&
	 		...&
	 		...&
	 		...\\
	 		0\ar[r]^{\bar \partial}&
			\Omega^{0, 2}(X)\ar[u]^{\bar\partial}\ar[r]^{\partial}&
	 		\Omega^{1, 2}(X)\ar[u]^{\bar\partial}\ar[r]^{\partial}&
	 		\Omega^{2, 2}(X)\ar[u]^{\bar\partial}\ar[r]^{\partial}&
	 		...\\        
	 		0\ar[r]^{\bar \partial}&
			\Omega^{0, 1}(X)\ar[u]^{\bar\partial}\ar[r]^{\partial}&
	 		\Omega^{1, 1}(X)\ar[u]^{\bar\partial}\ar[r]^{\partial}&
	 		\Omega^{2, 1}(X)\ar[u]^{\bar\partial}\ar[r]^{\partial}&
	 		...\\        
	 		0\ar[r]^{\bar \partial}&
			\Omega^{0, 0}(X)\ar[u]^{\bar\partial}\ar[r]^{\partial}&
	 		\Omega^{1, 0}(X)\ar[u]^{\bar\partial}\ar[r]^{\partial}&
	 		\Omega^{2, 0}(X)\ar[u]^{\bar\partial}\ar[r]^{\partial}&
	 		...\\
	 		0					  &
	 		0\ar[u]^{\bar\partial}&
	 		0\ar[u]^{\bar\partial}&
	 		0\ar[u]^{\bar\partial}\\
	}}
	\caption{The Dolbeaut Double Complex}
	\end{figure}
\end{lemma}
\begin{proof}
	In local coordinates for a differential form 
	\[\phi = \sum f_{IJ} dz_I\wedge d\bar z_J\]
	(where the cardinality of $I$ is $p$, the cardinality of $J$ is $q$, by
	$dz_I$ we mean an expression of the form $dz_{i_1}\wedge ... \wedge
	dz_{i_p}$, and by $f_{IJ}$ we mean a specific function for each choice of
	index in $I$ and in $J$), we can write
	\[\partial \phi = \sum_{i, I, J} \pd[f_{IJ}]{z_i}dz_i\wedge dz_I\wedge d\bar z_J\]
	And 
	\[\bar\partial \phi = \sum_{i, I, J} \pd[f_{IJ}]{\bar z_i}d\bar z_i\wedge dz_I\wedge d\bar z_J\]
	Note that $d = \partial + \bar\partial$ and the fact that the exterior
	derivative is linear implies $d^2 = \partial\partial +
	\partial\bar\partial + \bar\partial\partial + \bar\partial\bar\partial = 0$.
	But the images of $\partial\partial$, of $\partial\bar\partial +
	\bar\partial\partial$, and of $\bar\partial\bar\partial$ live in
	$\Omega^{(p+2, q)}$, $\Omega^{(p+1, q+1)}$, and $\Omega^{(p, q+2)}$,
	respectively, so this means $\partial\partial = \bar\partial\bar\partial =
	0$ and $\partial\bar\partial = -\bar\partial\partial$. In particular,
	$\ker(\partial)\subset\im(\partial)$, yielding a double complex. 
\end{proof}
\begin{lemma}[The Maximum Modulus Principle] 
	Let $f:U\subset \mathbb C^n\to \mathbb C$ be holomorphic and non-constant
	(for $U$ open). Then
	$|f(x)|$ does not attain a local maximum on $U$. 
\end{lemma}
\begin{corollary}
	Every holomorphic function on a compact manifold is constant. 
\end{corollary}
\begin{theorem}[Hartogs' Theorem]
Any holomorphic function on the compliment of an open disc extends to a function on the
	whole disk.	
\end{theorem}
\begin{remark}
	The proof of the above theorems is omitted for brevity.
\end{remark}
\subsection{Vector Bundles}
\begin{definition}
	A vector bundle is \textbf{holomorphic} if it's transition functions are
	holomorphic. 
\end{definition}
	\begin{definition}
		A \textbf{Hermitian inner product} on a complex vector space $V$ is a
		function $\langle \bullet, \bullet\rangle : V\times V\to \mathbb C$
		which is distributive over addition in each component and satisfies:
		\begin{itemize}
			\item $\langle \lambda a, b\rangle = \lambda\langle a, b\rangle$
				(Linear in the first component)
			\item $\langle a,\lambda  b\rangle = \bar\lambda\langle a, b\rangle$
				(Antilinear in the second component)
			\item $\langle a, b\rangle = \overline{\langle b, a\rangle}$
				(Conjugate - Symmetric)
			\item $\langle a, a\rangle\begin{cases}
					>0 & a\neq 0\\ =0 &a = 0
			\end{cases} $ (Positive Definite)
		\end{itemize}
	\end{definition}
	\begin{remark}
		The last condition is defined, as $\langle a, a\rangle =
		\overline{\langle a, a\rangle}$ is real. 
	\end{remark}
	\begin{definition}
		A \textbf{metric} on a vector bundle $E\to M$ is a smooth function $H:
		E\times E\to \mathbb C$ which is a Hermitian inner product on each fiber. 
	\end{definition}
	\begin{definition}
		Let $E$ be a vector bundle on $M$. We define the \textbf{differential
		$k$-forms with values in $E$} to be smooth sections of the vector bundle
		tensored with the $k$-th exterior power of the tangent bundle:
		$\Omega_\mathbb C^k(E)\coloneq C^{\infty}(M, E\otimes
			\Omega_{\mathbb C}^k(M))$. We define the $0$-forms with values in
			$E$ to be sections of $E$. 
	\end{definition}
	\begin{definition}
		Similarly, we define \textbf{differential forms of type $(p, q)$} to be 
		$\Omega^{(p, q)}(E)\coloneq C^\infty(M, E\otimes \Omega^{(p, q)})$,
		where $(0, 0)$-forms are sections of $E$.
	\end{definition}
	\begin{remark}
		Note that in the case when $L=\mathbb C\otimes M$ is the trivial complex line
		bundle, we recover the original notions of differential forms and $(p,
		q)$ forms. Thus, it will suffice to state and prove theorems for forms
		valued in line or vector bundles; these then carry to standard forms. 
	\end{remark}
	\begin{remark}
		The Hermitian metric product on any vector bundle acts naturally on wedge
		products of that bundle. Moreover, if $E$ is a Hermitian bundle over a
		Hermitian manifold, there is a natural Hermitian metric on $\Omega^{(p,
		q)}(E)$. 
	\end{remark}
\section{A (Very) Brief Supplement on Sheaves}
\subsection{Basic Definitions}
	Let $X$ be a topological space. We define the $\opens(X)$ to be the
	category who's objects are open subsets of $X$ and who's morphisms are
	inclusions; i.e., for any pair of open sets $U\subset V$ of $X$, there is an
	arrow in $\opens(X)$ from $U$ to $V$.  
\begin{definition}
	A \textbf{presheaf} $\mathscr F$ on a topological space $X$ with values in
	an abelian category $\mathscr A$ is a contravariant functor $\mathscr
	F\colon\opens(X)\to A$. 
\end{definition}
We call elements of $\mathscr F(U)$ sections of $\mathscr F$ over $U$. We call
$\mathscr F(U\to V)$ for $U\to V$ an inclusion in $\opens(X)$ a restriction, and denote it
by $\bullet|_{U}$.
\begin{definition}
	A \textbf{sheaf} on a topological space $X$ is a presheaf on $X$ satisfying the following axioms:
	\begin{enumerate}
		\item For any open cover $\{U_\alpha\}$ of an open set $U$ of $X$, for
			any $f, g\in \mathscr F(U)$, if for each pair
			of intersections $U_\alpha\cap U_\beta$ we have
			$f|_{U_\alpha\cap U_\beta} = g|_{U_\alpha\cap U_\beta}$, then $f=g$.
		\item For any open cover $\{U_\alpha\}$ of an open set $U$ of $X$, for
			any collection $f_\alpha\in \mathscr F(U_\alpha)$, if for each pair
			of intersections $U_\alpha\cap U_\beta$ we have
			$f_\alpha|_{U_\alpha\cap U_\beta} = f_\beta|_{U_\alpha\cap
			U_\beta}$, then there is a $f\in \mathscr F(U)$ with $f|_{U_\alpha}
			= f_\alpha$ for each $\alpha$.
	\end{enumerate}
\end{definition}
\subsection{Čech Cohomology}
	Let $M$ be a manifold and $\underline U := \{U_{\alpha}\}$ a locally finite open cover
	of $M$. (In general, we'll be working with complex holomorphic manifolds, but I
	believe that much of this generalizes, perhaps even to topological manifolds or even general
	topological spaces). We'll define a chain complex:
\begin{align*}
	C^0(\underline U, \mathscr F) = &\bigoplus_{\alpha_1}\mathscr F(U_{\alpha_1})\\
	C^1(\underline U, \mathscr F) = &\bigoplus_{\alpha_1\neq\alpha_2}\mathscr
	F(U_{\alpha_1}\cap U_{\alpha_2})\\
	&\vdots\\ C^k(\underline U, \mathscr F) = &\bigoplus_{\alpha_1\neq...\neq\alpha_k}\mathscr
F(U_{\alpha_1}\cap \dots \cap U_{\alpha_k})
\end{align*}
	Where $\bigoplus$ is the direct sum of abelian groups.
	A $p$-cochain, or an element $\sigma$ of $C^p$, can be written as 
	\[\sum_{\alpha_1\neq ... \neq \alpha_{p+1}}\sigma_{\alpha_1, ..., \alpha_{p+1}}\]
	We define a cochain
	map $\delta^p:C^{p}(\underline U, \mathscr F) \to C^{p+1}(\underline U,
	\mathscr F)$ by the rule
		\begin{equation}\label{cechmap}
			\delta^p(\sigma) = \sum_{\alpha_1 \neq ...\neq
			\alpha_{p+1}}\sum_{j=0}^{p+1}(-1)^j\sigma_{\alpha_1, ...,
			\hat{\alpha_{j}}, ..., \alpha_{p+1}} |_{U_{\alpha_j}}
		\end{equation}
		Where the hat denotes omission, and the vertical bar denotes restriction (as usual).
	Note that $\delta^p\circ\delta^{p-1}$ can be written as 
		\begin{equation*}\label{cechmap}
			\delta^p\circ\delta^{p-1}(\sigma) = \sum_{\alpha_1 \neq ...\neq
			\alpha_{p+1}}\sum_{j=0}^{p+1}(-1)^{j}\left(\sum_{k=0}^{p}(-1)^k\sigma_{\alpha_1, ...,
			\hat{\alpha_{k}}, ..., \hat{\alpha_{j}}, ..., \alpha_{p+1}}
			|_{U_{\alpha_k}\cap U_{\alpha_j}}\right)
		\end{equation*}
		Where the $\hat\alpha_k$ term denotes the omission of the $k$-th $\alpha_i$
		\textit{after} the $j$-th $\alpha_j$ is omitted - that is, omitting the
		$\alpha_k$ if $k<j$ and omitting $\alpha_{k+1}$ if $k\geq j$.
		Note that both the term $\sigma_{\alpha_1, ..., \hat\alpha_j, ...,
		\hat\alpha_k,...,\alpha_{p+1}}|_{U_{\alpha_k}\cap U_{\alpha_j}}$ and
		$\sigma_{\alpha_1, ..., \hat\alpha_k, ...,
		\hat\alpha_j,...,\alpha_{p+1}}|_{U_{\alpha_k}\cap U_{\alpha_j}}$ appear in
		the sum, the first with $\alpha_j$ omitted by the first $\delta$, the second
		by the second - however, the signs will always be opposite, so every term
		cancels, leaving zero. From this we note that the coboundary map is in fact a
		coboundary. We'll often suppress the top index of $\delta^n$, especially
		when working over general $n$. 
	
		With this complex is thus associated a cohomology: 
		\[\check H^{n}(\underline U, \mathscr F) = \frac{\Ker(\delta^n)}{\ima(\delta^{n-1})}\]
		However, this cohomology depends on the choice of open cover. 
	We now seek to make formal the notion of subdividing an open cover further and
	further, and taking the cohomology of that subdivision. For a topological space
	$X$, form a directed system $\textbf{Ref}(X)$ who's objects
	are open covers of $X$ and where there is a morphism
	$\{U_\alpha\}\to\{U'_\beta\}$ if and only if $\{U'_\beta\}$ is a refinement of
	$\{U_\alpha\}$. Now taking $X$ to be a complex holomorphic manifold $M$, we
	define
\begin{definition}
	The \textbf{Čech Cohomology} $\check H^\bullet(M, \mathscr F)$ of a manifold
	$M$ with coefficients in a sheaf $\mathscr F$ is defined to be the direct 
	limit of $H(\underline U, \mathscr F)$ for $\underline U$ in the directed
	system 
	$\textbf{Ref}(X)$, where $\underline U \to \underline V$ in
	$\textbf{Ref}(X)$ induces a morphism between $H(\underline U,\mathscr F)\to
	H(\underline V, \mathscr F)$ by the naturality of restriction.
\end{definition} 
This definition is extremely hard to work with directly, so we establish some
results to ease the passage from Čech cohomology relative to an open cover to
true Čech cohomology:
\begin{definition}A \textbf{good cover} of a topological space is an open cover
	where each open set in the cover and each finite intersection of open sets
	in the cover is contractible. 
\end{definition}
\begin{lemma}\label{good}
	Every smooth (real) manifold admits a good cover.
\end{lemma} 
\begin{corollary}
	Every smooth (real) manifold equipped with an atlas admits a good cover
	which is a refinement of that atlas.
\end{corollary} 
We record the following lemma and theorem without proof; for more detailed treatment, see
\autocite{stacks}, \autocite{Wei}, or \autocite{GrHa}.
\begin{lemma}
	Direct limits in the category of $R$-modules for a ring $R$ are exact.
\end{lemma}
\begin{theorem}[Leray]
	If $\underline U$ is a good cover, than $\check H^n(\underline U, \mathscr F) =
	\check H^n(M, \mathscr F)$.
\end{theorem} 
	Since all sheaves we'll be working with have coefficients in a category of
	$R$-modules (for some ring $R$, usually $\mathbb Z$ or a field) we'll be able to
	say that if the sequence 
	\[0\to H^\bullet(\underline U, \mathscr F)\to H^\bullet(\underline U, \mathscr
	G)\to H^\bullet(\underline U, \mathscr H)\to 0\]
	Is exact for every refinement of $\underline U$, the sequence
	\[0\to H^\bullet(M, \mathscr F)\to H^\bullet(M, \mathscr
	G)\to H^\bullet(M, \mathscr H)\to 0\]
	is likewise exact. 
	
	In particular, note that  for any good cover $\underline U$ (letting $H^{-1} =
	0$) we have that  $H^0(\underline U, \mathscr F)$ is 
	\[\frac{\{f_{\alpha} : f_{\beta}|_{U_\gamma} -
	f_{\gamma}|_{U_\beta}  = 0\}}{\{0\}}\]
	But these are exactly the sections which can be glued together to a global
	section; as such, $\check H^0(\underline U, \mathscr F) = \mathscr F(M)$ for
	any good  
	cover $\underline U$. Then lemma \ref{good} gives us that $\check H^0(M,
	\mathscr F) = \mathscr F(M)$ for every manifold $M$. 
	\subsubsection{List of Sheaves}
	For the reader's convenience, here is a list of the sheaves we will work
	with:
	\begin{itemize}
		\item The sheaf $\mathcal O(M)$ of holomorphic functions on a manifold $M$.
		\item The sheaf $\mathfrak M(M)$ of meromorphic functions on a manifold $M$.
		\item The sheaf $C^\infty(M)$ of smooth functions on a manifold $M$.
		\item The sheaf $\Omega_\mathbb R(M)$ of $\mathbb R$-valued differential
			$k$-forms on a manifold $M$. 
		\item The sheaf $\Omega(M) = \Omega_\mathbb C(M)$ of $\mathbb C$-valued
			differential $k$-forms on a manifold $M$. 
		\item The sheaf $\Omega^{(p, q)}(M)$ of forms of type $(p, q)$ on a
			manifold $M$.
		\item The sheaf $\mathcal O(M, E)$ of holomorphic sections of a vector bundle $E$ on a manifold $M$.
		\item The sheaf $\mathfrak M(M, E)$ of meromorphic sections of a vector bundle $E$ on a manifold $M$.
		\item The sheaf $C^\infty(M, E)$ of smooth sections of a vector bundle $E$ on a manifold $M$.
		\item The sheaf $\Omega_\mathbb C(E)\coloneq C^{\infty}(M, E\otimes
			\Omega_{\mathbb C}^k(M))$ of differential $k$-forms on a
			manifold $M$ with values in a vector bundle $E$. For $k=0$, we
			identify this with sections of $E$. 
		\item The sheaf $\Omega^{(p, q)}(E)\coloneq C^{\infty}(M, E\otimes
			\Omega^{(p, q)}(M))$ of differential $(p, q)$-forms on a
			manifold $M$ with values in a vector bundle $E$. For $p=q=0$, this
			is identified with sections of $E$. 
	\end{itemize}
%
	%
	%
	%
	%
\subsection{Sheaves and Line Bundles}
	Recall that a real line bundle is a vector bundle who's fibers are diffeomorphic
	to $\mathbb R$. We'll be working with complex (holomorphic) line bundles, or vector bundles
	who's fibers are biholomorphic to $\mathbb C$ with local holomorphic
	trivializations; from now on, ``line bundle'' will
	mean complex holomorphic line bundle. Let $\{U_\alpha,
	\phi_\alpha\}_{\alpha\in A}$ be local (holomorphic) trivializations for a
	line bundle $L$ over $M$ with projection $\pi:L\to M$. Let $f\alpha$ be a function in
	the multiplicative holomorphic structure sheaf $\mathcal O^*(U_\alpha)$ for some
	$\alpha$. Then, as multiplying a holomorphic function to a complex vector space
	by a holomorphic complex-valued function yields another holomorphic function, we
	can define new trivializations
	$\phi_{\alpha, f_\alpha}:= f\cdot \phi_\alpha$ (where the $\cdot$ is standard
	scalar multiplication). Note that the transition maps between the
	trivializations $\phi_{\alpha, f_\alpha}$ and $\phi_{\alpha, f'_\alpha}$ will be
	$\frac{f'_\alpha}{f_\alpha}$. Furthermore, if we consider two trivializations
	$\phi_\alpha, \psi_\alpha$ which define the same line bundle $L$, we can
	consider the (holomorphic) transition function
	$\phi_\alpha^{-1}\circ\psi_\alpha$. But on each of the fibers $\pi^{-1}(x)$ in
	this  neighborhood this transition function restricts to a linear function $f_x:\mathbb C\to
	\mathbb C$; since linear functions on $\mathbb C$ are just constant multiples.
	Moreover, $f_x(a)$ varies holomorphicly as a function of $x$ for all $a$ and is
	nonzero, meaning that this family of stalks defines a function $f_\alpha\in \mathcal
	O^*(U_\alpha)$. Given two open covers and two family of transition functions, we
	can thus see that the transition functions determine the same line bundle if
	they are `associates' by elements of $\mathcal O^*(U_\alpha)$.
 	
	The work that we're doing with open covers, stalks, and intersections hints at
	a deeper connection with Čech cohomology. Indeed, if we consider a line bundle
	$L\to M$ with trivializations $f_\alpha$ on an open cover $\{U_\alpha\}$ (and
	corresponding transition functions $f_{\alpha_1\alpha_2}$) we have on the
	intersection $U_\alpha\cap U_\beta$ that  $f_{\alpha\beta}f_{\beta\alpha} = 1$
	and on the intersection $U_\alpha\cap U_\beta\cap U_\gamma$ that
	$f_{\alpha\beta}f_{\beta\gamma}f_{\gamma\alpha} = 1$ (where juxtaposition is
	multiplication of transition functions). But consider the Čech coboundary
	operation acting on these cochains (using multiplicative notation):   
	\[\delta f_{\alpha\beta} = 
	(f_{\alpha_1\alpha_2}|_{U_{\alpha_1}\cap U_{\alpha_2}})
	(f_{\alpha_2\alpha_3}|_{U_{\alpha_2}\cap U_{\alpha_3}})
	(f_{\alpha_3\alpha_1}|_{U_{\alpha_3}\cap U_{\alpha_1}})\]
	This is clearly 1 (the identity) if and only if our two line bundle conditions
	hold. Moreover, we've just seen that two line bundles are the same if and only
	if they differ multiplicatively by an element of $\mathcal O^*(M)$. This yields
	that the set of all line bundles over a manifold $M$ is precisely $\check H^1(M,
	\mathcal O^*)$, as when we take further refinements we can restrict these
	trivializations and obtain the same results, so we can refine to a good cover by
	lemma \ref{good}.
	
\section{Line Bundles and Positivity}
	\begin{definition}
		A \textbf{connection} on a line bundle $L$ is a linear map $D:\Gamma(M,
		\Omega^0(L))
		\to \Gamma(M, \Omega^1(L))$ satisfying:
		\[D (f s) = s\otimes df + f D(s)\]
		Where $f\in C^\infty(M)$ is a function, $s\in \mathcal C^\infty_L(M)$ is a section of $L$, and
		$d$ is the exterior differential.
	\end{definition}
	\begin{remark}
		Pick an open set $U$ on which there exists a non-vanishing section $e$.
		Call $e$ a \textbf{frame} for $L$. Then we have that $D(e) = e\otimes
		\theta$
		for some function $\theta$. Moreover, $D$ can be recovered from the data
		$(\theta, e)$; note that for any connection $D$ and section $s$ we can
		write $s = fe$ for $f = s/e$, and we then have $D(s) = D(fe) =  e\otimes
		df + fD(e) = e\otimes df + fe\otimes\theta$. We may sometimes write the data
		$(\theta, e)$ as $\theta_e$.
		\label{rewrite}
	\end{remark}
	\begin{definition}
		A connection $D$ on a vector bundle $L\to M$ for complex manifold $M$
		is \textbf{compatible with the complex structure} of $M$ if
		$\bar\pi\circ D = \bar\partial$ (where $\bar\pi$ is projection onto
		$L\otimes \Omega^{0, 1}$). We define
		$D^{\bar\partial}\coloneq\bar\pi\circ D$ and 
		$D^{\partial}\coloneq\pi\circ D$; note that $D = D^\partial +
		D^{\bar\partial}$.
		\end{definition}
		\begin{definition}
		If the line bundle has a Hermitian metric, the connection is
		\textbf{compatible with the metric} $d$ if $d\langle f, g\rangle = \langle D f,
			g\rangle + \langle g,
		D f\rangle$, where $\langle,\rangle$ is the Hermitian inner product and $d$ is
			the exterior derivative.
	\end{definition}
	\begin{lemma}
		If a connection $D$ is compatible with the complex structure, then
		for a holomorphic frame $\theta$ is of
		type $(1, 0)$. If it is compatible with the metric, then for a
		orthonormal frame $\theta =
		-\overline{\theta}$\label{char}
	\end{lemma}
	\begin{proof}
		Suppose first $D$ is compatible with the complex structure, and let $e$
		be a holomorphic frame. But then $e\otimes\theta = D(e) = D^\partial(e) +
		D^{\bar\partial}(e) = D^\partial(e) +\bar\partial(e) = D^\partial(e)$,
		and $\theta$ has type $(1, 0)$. Moreover, if $D$ is compatible with
		the metric, then let $e$ be an orthonormal frame. Then $0 = d1 = d\langle e,
		e\rangle = \langle De, e\rangle + \langle e, De\rangle = \langle
		e\otimes \theta, e\rangle + \langle e, e\otimes \theta \rangle = \theta + \bar\theta$, and we
		are done. 
	\end{proof}
	\begin{lemma}
		For a holomorphic vector bundle $L$ with a Hermitian metric, there is a unique
		connection $D$ which is compatible with both the metric and the
		complex structure. We call this connection the \textbf{Chern
		Connection} or the \textbf{metric connection}.
	\end{lemma}
	\begin{proof}
		We begin by showing uniqueness, for which it suffices to work in a
		general open set $U$. Let $e$ be a holomorphic frame for $L$ over $U$.
		Then $\theta_e$ is of type $(1, 0)$. If we let $h = \langle e,
		e\rangle$, we can write $dh - d\langle e, e\rangle  = \langle De,
		e\rangle + \langle e, De\rangle  = \theta\langle e, e\rangle +
		\bar\theta\langle e, e\rangle = h\otimes\theta+ \bar h\otimes\theta$. But by Lemma
		\ref{char}, $\theta$ has type $(1, 0)$ and $\bar\theta$ has type $(0,
		1)$. By decomposing $d = \partial + \bar\partial$ and matching types, we
		obtain that $\partial h = h\otimes\theta$ and $\bar\partial h = h\otimes\bar\theta$.
		But solving both these equations yields $\theta = h^{-1}(\partial h)$,
		which shows uniqueness. It also shows existence locally; however,
		uniqueness and existence locally is enough to show existence globally,
		and we are done. 
	\end{proof}
	Given a $D\colon \Gamma(M, \Omega^0(L))\to \Gamma(M, \Omega^1(L))$ we can induce
	a map \[D\colon \Gamma(M, \Omega^0(L))\to \Gamma(M, \Omega^2(L))\] by
	simply specifying that $D$ is linear and that Leibniz's rule holds for a section $s$ and a form
	$\sigma$:
	\[D(s\otimes \sigma) = s\otimes d\sigma + D(s)\wedge\sigma\]
	We then obtain an operator $D^2\colon \Gamma(M, \Omega^0(L))\to \Gamma(M,
	\Omega^2(L))$ by $D^2= D\circ D$. 
	\begin{definition}
	We call this operator the \textbf{curvature} of the connection $D$. 
	\end{definition}
\begin{remark}
	Similarly to Remark \ref{rewrite}, for any frame $e$ we can write $D^2(e) =
	e\otimes\Theta$ for some $2$-form $\Theta$.
\end{remark}
\begin{lemma}[Simplified Cartan Structure Equation]
	If $D^2$ is the curvature of a connection $D$ on a line bundle $L$, $e$ is a
	local frame for $L$, and $D^2(e) = e\otimes\Theta$, then locally we have
	\[\Theta = d\theta\]
\end{lemma}
\begin{proof}
	Pick a local frame $e$ for $L$ and consider $e\otimes\Theta = D^2(e) =
	D(e\otimes \theta) = e\otimes d\theta + D(e)\wedge \theta = e\otimes
	d\theta + e\otimes\theta\wedge\theta$. This holds for any frame and
	$\theta\wedge\theta = 0$, so we reach
	the intended conclusion.
\end{proof}
\begin{remark}
	A real $(1, 1)$ form $\omega$ is a complex $(1, 1)$ form with real
	coefficients satisfying $\bar\omega = \omega$. 
\end{remark}
	\begin{remark}
			Observe that for a smooth function $f$ and sections $s, t$ (writing $D(s)$
			as $s'\otimes \sigma$)
			\begin{align*}
				D^2(fs+t) =& D(s\otimes df +f (s'\otimes \sigma)) + D^2(t)\\
				=&s\otimes d^2f + s'\otimes \sigma\wedge df + (fs')\otimes d\sigma + D(fs')\wedge \sigma
				+ D^2(t)\\ 
				=&s\otimes d^2f + s'\otimes \sigma\wedge df + (fs')\otimes d\sigma + (s'\otimes df +
				fD(s'))\wedge \sigma
				+ D^2(t) \\
				=&s\otimes d^2f + s'\otimes \sigma\wedge df + (fs')\otimes d\sigma + s'\otimes df\wedge \sigma +
				fD(s')\wedge \sigma + D^2(t) \\
				=& (fs')\otimes d\sigma + fD(s')\wedge \sigma + D^2(t) \\
				=& f(s'\otimes d\sigma + D(s')\wedge \sigma) + D^2(t) \\
				=& f(D(s'\otimes \sigma)) + D^2(t) \\
				=& f(D^2(s)) + D^2(t) \\
			\end{align*}
				From this we see that $D^2$ respects multiples by a smooth
				function. Moreover, we have $D^2(s) = D^2(fe)
	=fD^2(e) = fe\otimes \Theta$.\label{linear}
			\end{remark}
\begin{remark}\label{invariant}
	Note that if $e$ and $e'$ are frames for $L$ on the same open set with $g =
	\frac{e}{e'}$, then
	$D^2(e') = D^2(ge)$. By Remark \ref{linear}, we observe that this is the
	same as $gD^2(e) = g(e\otimes\Theta_{e})$. But then this is the same as
	$g((g^{-1}e')\otimes \Theta_e)$, so we see $\Theta_{e'} = gg^{-1}\Theta_e =
	\Theta_e$. This is in direct contrast to $\theta_e$. 
\end{remark}
\begin{corollary}
	Suppose $D$ is a connection which is compatible with a metric and a complex
	structure.
		The curvature form $\Theta$ is an imaginary $(1, 1)$-form; that is,
		$iD^2(\sigma)$ lies in $\Omega^2_\mathbb R(E)\cap \Omega^{(1, 1)}(E)$ for
		any $\sigma$.
\end{corollary}
\begin{proof}
	First pick a holomorphic local frame $e$ for $L$. Then $\theta_e$ is a $(1,
	0)$ form. But then $\Theta_e = d\theta_e = \partial\theta_e +
	\bar\partial\theta_e$. Moreover, $\bar\partial\theta$ must be of type $(1,
	1)$ (as $\theta$ is of type $(1, 0)$). But $\partial$ only increases the
	holomorphic type (the first coordinate); as such, $\Theta_e$ has no $(0, 2)$
	part. But Remark \ref{invariant} states that $\Theta = \Theta_e$ is
	invariant under change of local frame; in
	particular, picking a holomorphic frame $e'$, we see that $\overline\Theta =
	\overline{d\theta_{e'}} = d\overline{\theta}_{e'} = -d\theta_{e'}$, so in
	particular $\Theta = -\overline{\Theta}$ (for any frame, as $\Theta_e$ is
	indpendent of $e$ by Remark \ref{invariant}). But then $\overline{\Theta}$ must
	likewise have no $(0, 2)$ part, and thus be purely of type $(1, 1)$.
	Moreover, the condition $\Theta = -\overline{\Theta}$ means that
	$\overline{i\Theta} = i\Theta$, and $i\Theta$ is \textbf{real} in the
	sense below.
\end{proof}
\begin{lemma}\label{add}
	If $L$ and $L'$ are Hermitian bundles with Chern connections $D$ and $\tilde
	D'$ and curvature forms $\Theta$ and $\Theta'$, then on $L\otimes L'$ the
	curvature form is given by $\Theta +\Theta'$.
\end{lemma}
\begin{proof}
	Note first that $L\otimes L'$ has a natural Hermitian structure given by
	(if $h(x, y)$ is the Hermitian metric on $L$ and $h'(x, y)$ the metric on
	$L'$):
	\[\langle x\otimes x', y\otimes y'\rangle = h(x, y)h'(x', y')\]
	Clearly this is bilinear in both arguments and thus a well defined function;
	the Hermitian axioms follow naturally. Note that a connection $\Delta$ which is
	compatible with this metric satisfies:
	\[d\langle x\otimes x', y\otimes y'\rangle = \langle \Delta(x\otimes x'),
	y\otimes y'\rangle + \langle x\otimes x',
	\Delta(y\otimes y')\rangle\]
	For local frames $e, e'$ with $h(e, e) = h'(e', e') = 1$ we have:
	\[0 =d\langle e\otimes e', e\otimes e'\rangle = \langle \Delta(e\otimes e'),
	e\otimes e'\rangle + \langle e\otimes e',
	\Delta(e\otimes e')\rangle\]
	\[= \langle (e\otimes e')\otimes\theta_e, e\otimes e'\rangle + \langle e\otimes e',
	(e\otimes e')\otimes\theta_{e'}\rangle\]
	\[=\theta_{e} + \theta_{e'}\]
	The result follows from the Cartan structure equation, the linearity of the
	differential, and the fact that the curvature form is invariant under change
	of frame.
\end{proof}
\begin{remark}
	To a Hermitian metric we can associate a real $(1, 1)$-form $\omega \coloneq
	\frac i2 \langle\bullet, \bullet\rangle$. Conversely, to any $(1,
	1)$ form $\omega$ with $\bar\omega = \omega$ we can associate an operator
	$(v, w) \coloneq -2i\omega(v, \bar w)$. Given tangent vectors $v, w$ we have
	\[\overline{-2i\omega(v, \bar w)} = \overline{-2i\sum v_idz_i\wedge \bar
	w_id\bar z_j} = -\bar{2i}\sum \bar{v_i}d\bar z_i\wedge 
	w_idz_j= -2i\omega(w, \bar v),\]
	and given a function $f$ we have 
	\[2i\omega(fv, \bar w) = 2i\sum fv_idz_i\wedge \bar w_id\bar z_j =
	2if\omega(v, \bar w)\]
	And
	\[2i\omega(v,\overline{fw}) = 2i\sum v_idz_i\wedge \overline{fw}_idz_j =
	-2i\bar{f}\omega(w, \bar v)\]
	So $\omega$ induces a Hermitian operator. If this operator is
	positive-definite; that is, if it is a Hermitian inner product, then we call
	$\omega$ \textbf{positive}. Note that a sufficient condition for this is
	that $-i\omega(v, \bar v)(x)$ is positive for all $x$ and all $v$.
\end{remark}
\begin{definition}
	A \textbf{Kähler manifold} is a manifold with a positive, $d$-closed $(1,
	1)$-form $\omega$; equivalently, a manifold with a metric who's associated
	form is closed. 
\end{definition}
\begin{definition}
A line bundle $L$ is \textbf{positive} if 
	$\frac{i}{2\pi}\Theta$ is a positive $(1, 1)$ form. 
\end{definition}
\section{The Kodaira Vanishing Theorem}
\subsection{Some Harmonic Theory}
We now assume $M$ is a compact manifold.
\begin{definition}
	Define the \textbf{Hodge star operator} $\star:\Omega^{p, q}(M)\to \Omega^{n-p,
	n-q}(M)$ to be the operator such that $\alpha \wedge\star\beta = \langle
	\alpha, \beta\rangle\Phi$, where $\Phi$ is the volume form $(n!)\omega$. We
	can extend this to an operator $\star_E$ on $\Omega^{\bullet,
	\bullet}(E)$ for a vector bundle $E$ by the rule $\star_E(s\otimes
	\sigma) = s\otimes \overline{\star\sigma} = s\otimes\star\overline\sigma$
	(as $\star$ is $\mathbb C$-antilinear). 
\end{definition}
\begin{remark}
	Suppose $M$ is a compact manifold with a metric on the tangent bundle and
	$L$ a line
	bundle with a metric; we then have metrics on $\Omega^k(L)$. 
	Given an operator $D^{\bar\partial} = \bar\partial$ on a line bundle $L$
	with a Hermitian metric. We define an inner product $( , )$ on $\Omega^{(p,
	q)}(L)$ by
	\[(\alpha, \beta) = \int_{M} \langle \alpha, \beta\rangle \Phi\]
	Where $\Phi$ is a volume form.
\end{remark}
\begin{definition} 
	The adjoint $\bar\partial^*$ of $\bar\partial$ is the operator satisfying:
		\[(\bar\partial \phi, \psi) = (\phi, \bar\partial^*\psi)\] 
\end{definition}
\begin{remark}
	This operator is given by $\star\bar\partial\star$, showing it's existence.
\end{remark}
\begin{definition}
	We define the $\bar\partial$\textbf{-Laplacian} to be $\Delta =
	\bar\partial\bar\partial^* + \bar\partial^*\bar\partial$. A form is called
	\textbf{$\bar\partial$-harmonic} if it is in the kernel of this Laplacian.
	We define $\mathcal H^{p, q}$ to be the set of all harmonic $(p, q)$ forms.
\end{definition}
\begin{definition}
	Recall that $\bar\partial$ is a well defined operator on $\Omega^{(p, q)}$
	which squares to zero. We define the \textbf{Dolbeault cohomology with
	coefficients in $L$} to be $H^{(p, q)}(L) \coloneq H^q(\Omega^{(\bullet,
	q)}(L), \bar\partial)$ (where $H^q$ is the standard cohomology functor).
\end{definition}
\begin{theorem}[The Hodge Theorem] 
	The following results hold:
	\begin{enumerate}
	\item The space $\mathcal H^{(p, q)}(L)$ is finite-dimensional.
	\item There is an isomorphism $\check H^{q)}(M, \Omega^p(L)) \cong \mathcal H^{(p, q)}(L)$. 
	\item In fact, harmonic forms are representatives of $H^{(p, q)}$. 
	\end{enumerate}
\end{theorem}
\begin{definition}
	If $M$ is Kähler with form $\omega$, we define the operator $\lambda:\Omega^{(p, q)}(L) \to \Omega^{(p+1,
	q+1)}(L)$ to be $\lambda(s\otimes\eta) = s\otimes\omega\wedge\eta$ and $\Lambda$
	to be the adjoint of $\lambda$. 
\end{definition}
\begin{definition}
	We define the \textbf{commutator} $[A, B]\coloneq A\circ B - B\circ A$.
\end{definition}
\subsection{Identities from Hodge Theory}
We record here (some proven, some cited) a few useful identities which we will
leverage in our proof of the Kodaira Vanishing Theorem.
\begin{lemma}\label{woah}
	$[\Lambda ,\lambda]|_{\Omega^{p, q}(L)} = (n-(p+q))\id$. 
\end{lemma}
\begin{remark}
	This is a result from linear algebra. See \autocite{Huy} Prop. 1.2.26.
\end{remark}
\begin{lemma}[Hodge identities]
	\[[\Lambda, \bar\partial] = -\frac{i}2\partial^*\]
	\[[\Lambda, \partial] = -\frac{i}2\bar\partial^*\]
\end{lemma}
\begin{proof}
	C.F. \autocite{GrHa}.
\end{proof}
\begin{lemma}[Nakano Identity]\label{nak}
	On $\Omega^{\bullet, \bullet}(L)$ for a line bundle $L$ the operator
	$\bar\partial$ is well defined, and we have: 
	\[[\Lambda, \bar\partial] = -\frac{i}2(D^{\partial})^*\coloneq
	i(\star_{E^*}\circ D^{\partial}_{E^*}\circ \star_{E})\]
	Where $D_{E^*}$ is the connection on $E^*$ induced by the Chern connection
	$D$ on $E$. 
\end{lemma}
\begin{proof}
	 We can verify this property locally. Begin with a section $\eta$, by
	 writing $\eta = e\otimes n$ for a local frame $e$. Then $\bar\partial \eta
	 = D^{\bar\partial}(\eta)= e\otimes \bar\partial n + ne \otimes \theta^{(0,
	 1)}$, where $\theta^{(0, 1)}$ is the antiholomorphic portion of $\theta$.
	 Furthermore, we have $\Lambda\eta = e\otimes \Lambda(n)$. Then
	 \[[\Lambda, \bar\partial]\eta = e\otimes [\Lambda, \bar\partial](n) +
	 e\otimes [\Lambda, \theta^{(0, 1)}](\eta)\]
	 By the Hodge identities, then,
	 \[[\Lambda, \bar\partial]\eta = e\otimes \left(-\frac i2\right)
	 \partial^*n+
	 e\otimes [\Lambda, \theta^{(0, 1)}](\eta)\]
	 But we can likewise write $D^{\partial}(\eta)$ as $e\otimes \partial n +
	 ne\otimes \theta^{(1, 0)}$;  
	 ``taking the adjoint'' of this equivalent characterization for $D$, (which works as the Hodge $\star$
	 is linear) yields
	 \[(D^{\partial})^*(\eta) = e\otimes \partial^* n + \eta\otimes (\theta^{(1,
	 0)})^*\]  
	 Then 
	 \[[\Lambda, \bar\partial](\eta) + \frac i2 (D^{\partial})^*(\eta) =  e\otimes
	 \left(-\frac i2\right)\partial^*n +
	 e\otimes [\Lambda, \theta^{(0, 1)}](\eta)+ \left(-\frac i2\right)e\otimes \partial^* n + \left(-\frac i2\right)\eta\otimes (\theta^{(1,
	 0)})^*\]
	 \[=e\otimes[\Lambda, \theta^{(0, 1)}](\eta) + \left(-\frac i2\right)\eta\otimes(\theta^{(1, 0)})^*\]
	 But the left hand side of this equation is globally and intrinsically
	 defined; as such,the right is as well. Picking local coordinates such that
	 the right hand side vanishes at a  point (which we can do for any point),
	 we see that 	 \[[\Lambda, \bar\partial] + \frac i2 (D^{\partial})^*  =
	 0,\] and we are done. 
\end{proof}
\begin{remark}
	Note that $D = D^{\partial}+ D^{\bar\partial}$, where
	$D^{\bar\partial}$ satisfies $D^{\bar\partial}(s\otimes\sigma) =
	\sigma\otimes\bar\partial s + s\wedge D^{\bar\partial}\sigma$. But
	$D^{\bar\partial} = \bar\partial$ by  compatibility, so
	$D^{\bar\partial}(s\otimes \sigma) = \sigma\otimes \bar\partial s +
	s\wedge \bar\partial \sigma$. But this is just the product rule for $\bar\partial$,
	and we have that $D^{\bar\partial} = \bar\partial$. 
\end{remark}
\begin{remark}\label{otherone}
	The identity $[\Lambda, \bar\partial] = -[\bar\partial, \Lambda]$ implies
	that $[\bar\partial, \Lambda] = \frac i2(D^{\partial})^*$. 
\end{remark}
\subsection{Statement and Proof}
\begin{theorem}[Kodaira-Nakano Vanishing Theorem]
	Let $M$ be a compact, Kähler manifold and $L\to M$ a positive line bundle.
	Then the cohomology
		\[\check H^q(M, \Omega^p(L)) = 0\]
		Whenever $p+q>n$. 
\end{theorem}
\begin{proof}
	By assumption, $\frac{i}{2\pi}\Theta$ is a positive $(1, 1)$ form; thus it induces a
	Hermitian metric on $M$, which we fix. Furthermore, we then have
	$\lambda$ is wedging with $\frac{i}{2\pi}\Theta$ by definition. By
	Hodge, it suffices to show that there are no
	nonzero harmonic forms in $\mathcal H^{(p,q)}(L)$ for $p+q>n$. 
	To this end, let $\eta\in
	\mathcal H^{(p, q)}(L)$ be a harmonic form. Since $\eta$ is a representitive
	of $H^{(p, q)}(E)$, it is $\bar\partial$-closed. 

	Let $e$ be a frame for $L$.
	Then we have $D^2(\eta) = D(D^\partial(\eta)) + D(\bar\partial(\eta))$;
	but the second term goes to zero. We can decompose $D$ into $D^{\partial}$
	and $D^{\bar\partial}$ in the same way; we then have:
	\[D^2(\eta) = D^{\partial}(D^{\partial}(\eta)) + D^{\partial}(D^{\bar\partial}(\eta))
	+ D^{\bar\partial}(D^{\partial}(\eta)) +
	D^{\bar\partial}(D^{\bar\partial}(\eta))\]
	But recall that $D$ is compatible with the complex structure, so 
	\[D^2(\eta) = D^{\partial}(D^{\partial}(\eta)) + D^{\partial}({\bar\partial}(\eta))
	+ D^{\bar\partial}(D^{\partial}(\eta)) +
	D^{\bar\partial}({\bar\partial}(\eta))\]
	Since $D^2$ has type $(1, 1)$ and by compatibility, we have 
	\[D^2(\eta) =\bar\partial(D^{\partial}(\eta)) +
	D^{\partial}(\bar\partial)(\eta) \]
	And as $\eta$ is closed,
	\[D^2(\eta) = {\bar\partial}(D^{\partial}(\eta)) \]
	Meaning $\eta\otimes\Theta = D^{\bar\partial}(D^{\partial}(\eta))$, so 
	\[2i((\Lambda\Theta)\eta, \eta) = 2i(\Lambda
	\bar\partial D^{\partial}(\eta), \eta)\]
	By Lemma \ref{nak}, this is:
	\[=2i\left(\left(\bar\partial\Lambda- \frac i2
	(D^{\partial})^*\right)D^{\partial}\eta, \eta\right)\]
	We can rewrite $(\bar\partial\Lambda D^{\partial}\eta, \eta)$ as
	$(\Lambda D^{\partial}\eta, (\bar\partial)^*\eta)$. But
	$(\bar\partial)^*\eta = 0$, as \[0 = (\Delta \eta, \eta) =
	(\bar\partial\bar\partial^* \eta, \eta) + (\eta, \bar\partial^*\bar\partial)
	= (\bar\partial^*\eta, \bar\partial^*\eta) + (\bar\partial\eta,
	\bar\partial\eta) \]
	Implies (by the positive-definiteness of the inner product we have defined)
	that both $\bar\partial^*\eta$ and $\bar\partial\eta$ are 0. 
	But then we have
	\[2i\left(\left(\bar\partial\Lambda- \frac i2
	(D^{\partial})^*\right)D^{\partial}\eta, \eta\right)
	=2i\left(-\frac i2 (D^{\partial})^*D^{\partial}\eta, \eta\right)
	=\left((D^{\partial})\eta,D^{\partial} \eta\right) \geq 0\]
	By the positive-definiteness of the inner product. Similarly,
	\[2i(\Theta\Lambda\eta, \eta) = 2i(D^{\partial}\bar\partial\Lambda\eta,	
	\eta) + 2i(\bar\partial D^{\partial}\Lambda(\eta), \eta)\]
	The right most term is $2i(D^{\partial}\Lambda(\eta), \bar\partial^*\eta)
	=0$, and so again by Lemma \ref{nak},
	\[= 2i(D^{\partial}\left(\Lambda\bar\partial + \frac
	i2(D^\partial)^*\right)\eta, \eta) \]
	Since $\Lambda\bar\partial \eta$ vanishes, we have
	\[= 2i\left(D^{\partial}\frac
	i2(D^\partial)^*, \eta\right) \]
	By definition of adjoints:
	\[= -\left( (D^\partial)^*\eta, (D^{\partial})^*\eta\right) \leq 0\]
	But then we have that 
	\[2i(\Lambda\Theta\eta, \eta) - 2i(\Theta\Lambda\eta, \eta)\geq 0\]
	Or equivalently
	\[2i([\Lambda, \Theta], \eta) \geq 0\]
	But recall that $\Theta = \frac{2\pi}i \lambda$, so we obtain
	\[4\pi([\Lambda, \lambda]\eta, \eta) \geq 0\]
	But by Lemma \ref{woah}, we obtain
	\[4\pi((n-(p+q)\eta, \eta) \geq 0\]
	But $(\eta, \eta)$ is positive unless $\eta = 0$; as such, we obtain that $p+q=n$ implies
	$\eta =0$. 
\end{proof}

This concludes our proof of the Kodaira vanishing theorem. We now discuss a
notion from algebraic geometry, blowing up, which we will use as a way to turn a
point on a manifold into a divisor and thereby apply the theory of positive line bundles.
\section{Divisors}
	\begin{definition}An \textbf{analytic subvariety} of a manifold $M$ is a subset
		$V$ of $M$ where for every point $p$ in $V$ there is a neighborhood $U$ of
		$p$ in $V$ which is the zero set of a finite collection of holomorphic
		functions. A \textbf{smooth point} on an analytic subvariety $V$ of a manifold
		$M$ is a point with a neighborhood which is a submanifold of $M$. The set of
		all smooth points of $V$ is denoted $V^*$. 
	\end{definition}
	\begin{definition}
		An \textbf{analytic hypersurface} is an analytic subvariety defined locally by one
		function. 
	\end{definition}
	We use a definition of irreducible which is not quite standard, but equivalent
	to the standard definition:
	\begin{definition}
		An analytic variety $V$ is \textbf{irreducible} if $V^*$ is connected. 
	\end{definition}
	If an analytic variety is irreducible, it's \textbf{dimension} is the dimension
	of $V^*$ as a complex manifold. If a general analytic variety has dimension $k$
	irreducible components, we say it is of dimension $k$. A codimension $k$
	subvariety of an $n$ dimensional manifold is a dimension $n-k$ subvariety. 
	\begin{definition}
		If $V$ is an irreducible analytic hypersurface of $M$ with locally defining
		functions $f_\alpha$, and $g$ is a holomorphic function on $M$, we define the
		\textbf{order} $\ord_{V,p}(f)$ of $f$ along $V$ at a point $p$ on $M$ to be
		the greatest integer $a$ such that $g_\alpha = f^ah$ for $g_\alpha$ a
		defining function of $V$ at $p$. 
	\end{definition}
	A standard result from complex analysis states that this order is independent of
	the choice of $f_\alpha$ or of $p$; as such, we drop the $p$ and focus solely on
	the order of a function along a hypersurface. It follows that $\ord_V(fg)=
	\ord_V(f)+\ord_V(g)$. We extend this notion to
	meromorphic functions; if $f$ can be written locally as $\frac gh$ for $g, h$
	holomorphic and relatively prime, we define
	$\ord_V(f) = \ord_V(g) - \ord_V(h)$. 
	\begin{definition}
		A \textbf{divisor} on a manifold $M$ is a locally finite formal linear
		combination with integer coefficients of
		irreducible analytic hypersurfaces of $M$. 
	\end{definition}
	Recall that a codimension 1 subvariety $V$ is defined locally as the locus of a
	single holomorphic function; we can construct an open cover $\{U_\alpha\}$ of $M$ with functions
	$f_\alpha$ such that $V\cap U_\alpha = \{x\in U_\alpha | f_\alpha(x) = 0\}$. 
	As such, given a divisor 
		\[D = \sum_{\{V_i\}}a_iV_i\]
	By taking refinements of the open cover $\{U_\alpha\}$ associated to each $V_i$
	we can obtain an open cover $\{U_{\alpha}\}$ together with functions
	$\{g_{\alpha, i}\}$ where $V_i\cap U_{\alpha} =
	\{x\in U_{\alpha} | g_{\alpha, i}(x) = 0\}$. We call these functions
	\textbf{local defining functions} for the divisor $D$. Recall that the quotient
	sheaf $\mathfrak M^*/\mathcal O^*$ is the sheaf of functions who
	can be written locally as ratios of holomorphic functions, where we identify
	meromorphic functions which differ by a holomorphic multiple. We let 
	\[f_\alpha = \prod g_{\alpha, i}^{a_i}\]
	Which is a meromorphic function over $U_{\alpha}$. But this defines everywhere a
	function given locally as a ratio of two holomorphic functions; furthermore, on
	the intersections these functions differ only by a multiple of a holomorphic
	function. As such, we can
	glue together these $f_\alpha$ to obtain a global section of the quotient sheaf
	$\mathfrak M^*/\mathcal O^*$. 
	
	Conversely, for $\{f_\alpha\}$ a global section of the quotient sheaf $\mathfrak
	M^*/\mathcal O^*$ written locally as meromorphic functions modulo holomorphic
	multiples, we can associate to $f$ a divisor:
	\[D = \sum_{V_i\subset M}\ord_V(f_\alpha)V_i\]
	Where $V_i$ is an irreducible hypersurface of $M$.
	If there's any other representative $\{g_\beta\}$ of the same section of the
	quotient sheaf, we then have $f_\alpha/g_\beta\in\mathcal O^*(U_\alpha)$. But
	if this is the case, then $\ord(f_\alpha) - \ord(g_\beta) = 0$, because the
	order of any element of $\mathcal O^*(U)$ along any irreducible hypersurface $V$
	for any open set $U$ is zero (as in order for a holomorphic
	function to have a holomorphic inverse on an open set $U$, it must never be zero
	on the set $U$; in particular, if $f$ vanishes on any point of $U$, that
	holomorphic function can never be written as $hf^k$ for $k>0$, so the order must
	be zero). Equivalently, $\ord(f_\alpha) = \ord(g_\beta)$, and so $D$ is
	well-defined given $\{f_\alpha\}$.  
	
	This process works just as well for meromorphic functions on the manifold $M$; given
	a meromorphic function $f$, we define
	\[(f) =\sum_{V_i\subset M}\ord_V(f_\alpha)V_i\] 
	
	We have now defined a divisor given a section of $\mathfrak
	M^*/\mathcal O^*$, and conversely, associated a section of $\mathfrak
	M^*/\mathcal O^*$ to a divisor. These operations are also
	inverses of each other; letting $D$ be a divisor, taking $f_\alpha$ to be the
	product of the $g_{\alpha, i}^{a_i}$ where the $g_{\alpha, i}$ are defining
	functions of the hypersurfaces in the variety and the $a_i$ are the
	multiplicities of the hypersurfaces in the variety, we obtain that the divisor
	associated to this $f_\alpha$ is exactly the sum over all the irreducible
	hypersurfaces $V_i$ of $M$ of the order of $f_\alpha$ along $V$ times $V_i$; the
	order of $f_\alpha$ along $V$ is positive if and only if $f_\alpha$ is zero
	along $V$, with multiplicity $a_i$, and negative if and only if $1/f_\alpha$ is
	zero along $V$, with multiplicity $-a_i$. From this we recover the original divisor
	$D$. 
	
	We can describe a similar correspondence between divisors and line bundles. Let
	$D$ be a divisor with locally defining functions $f_\alpha$. Then the functions 
	$g_{\alpha\beta} = \frac{f_\alpha}{f_\beta}$ are holomorphic and nonzero on the
	intersection $U_\alpha\cap U_\beta$; furthermore, we have by definition the
	cocycle conditions
	\[ \begin{cases}
			g_{\alpha\beta}\cdot g_{\beta\alpha} =
			\frac{f_\alpha}{f_\beta}\frac{f_\beta}{f_\alpha} = 1\\
			g_{\alpha\beta}\cdot g_{\beta\gamma}\cdot g_{\gamma\alpha} =
			\frac{f_\alpha}{f_\beta}
			\frac{f_\beta}{f_\gamma}
			\frac{f_\gamma}{f_\alpha} = 1
		\end{cases}
	\]
	Which are exactly the conditions for transition maps for a line bundle, which
	(as we saw above) define a unique line bundle. Furthermore, if $f'_\alpha$ is a
	different locally defining function for $D$, then we define new
	$g'_{\alpha\beta} = \frac{f'_\alpha}{f'_\beta}$. However, these alternate
	$g'_{\alpha\beta}$ will differ from $g_{\alpha\beta}$ by a holomorphic function
	with an inverse, and thus define the same line bundle by the discussion above. 
	We call this well-defined line bundle the \textbf{associated line bundle} to the
	divisor $D$, and denote it $[D]$. 
	
	It is immediate from this description that for any divisors $D$ and $D$ defined
	locally by $f_{\alpha}$ and $f'_{\alpha}$, we have that $D+D$ is defined
	locally on $U_{\alpha_1}\cap U_{\alpha_2}$ by $f_{\alpha_1} f'_{\alpha_2}$, and
	so $[D+D]$ is given by $\frac{f_{\alpha_1}f_{\alpha_2}}{f_{\beta_1}f_{\beta_2}}$. But
	the line bundle $[D]\otimes[D]$ is given by the transition functions
	$\frac{f_{\alpha_1}}{f_{\beta_1}}\otimes \frac{f_{\alpha_2}}{f_{\beta_2}} 
	=\frac{f_{\alpha_1}}{f_{\beta_1}}\frac{f_{\alpha_2}}{f_{\beta_2}}$, and so
	$[\bullet]$ is a group homomorphism into the \textbf{Picaird group}, the group
	of line bundles under the tensor product.  
	
	Our next goal is to adapt this machinery from the case of functions on $M$ to
	the case of line bundles and sections of line bundles. Recall that, given any
	line bundle $L$, we denote by $\mathcal O(L)$ the sheaf of holomorphic sections
	of $L$. This has the structure of a $\mathcal O$-module; clearly, it is an
	abelian group (sections can be added pointwise); moreover, pointwise
	multiplication by a holomorphic function gives another holomorphic section. We
	base-change this module by tensoring with the sheaf of meromorphic functions,
	also an $\mathcal O$-module, to obtain the meromorphic sections of $L$:
	\[\mathfrak M(L) \coloneq \mathcal O(L)\otimes_{\mathcal O} \mathfrak M\]
	Moreover, just as holomorphic sections on an open set $U$ can be specified locally by
	holomorphic functions $s_\alpha
	\colon U\cap U_\alpha\to  \mathbb C$ with $s_\alpha = g_{\alpha\beta}s_\beta$ on
	$U\cap U_\alpha\cap U_\beta$ for $g_{\alpha\beta}$ an appropriate (holomorphic) transition function,
	meromorphic sections on $U$ can be specified locally as meromorphic functions
	$s_\alpha\colon U\cap U_\alpha \to \mathbb C$, related by
	holomorphic transition functions. From this it follows that the quotient of two
	local definitions of meromorphic functions is a holomorphic function, meaning
	that the order of a meromorphic section along an irreducible hypersurface is well-defined and invariant globally. As such, we
	define the order of a meromorphic section $s$ along an irreducible hypersurface $V$
	to be $\ord_V(s) = \ord_V(s_\alpha)$, where $U_\alpha$ intersects $V$. This
	allows us to define divisors associated to sections of line bundles; given a
	section $s$ of a line bundle $L$, we define a divisor:
	\[(s) = \sum_{V\subset M}\ord_V(s) V\]
	Where the sum is taken over all irreducible hypersurfaces $V$ of $M$. It follows
	immediately that $s$ is holomorphic if and only if $(s)$ is effective, and
	nonvanishing if and only if $(s)$ is zero.
	
	Suppose a divisor $D$ (or a section of the quotient sheaf $\mathfrak M/ \mathcal
	O$)
	is given locally by meromorphic defining functions $f_\alpha\in
	\mathfrak M(U_\alpha)$ for an appropriate open cover. Recall that $[D]$ is the
	line bundle given by transition functions which are ratios of the locally
	defining functions. Then each $f_\alpha$ is a meromorphic function $f_\alpha
	\colon U_\alpha\to \mathbb C$, and moreover, the $f_\alpha$ are related by transition
	functions $\frac{f_\alpha}{f_\beta}$ which are the holomorphic transitions of
	the line bundle. This gives a meromorphic section $s_f$ of $[D]$. Moreover, the
	divisor $(s_f)$ is 
	\[\sum_{V\subset M} \ord_V(s_f)V\]
	But $s_f$ has order $k$ along a codimension 1 hypersurface if and only if $f$
	has order $k$ along that same hypersurface, so this is exactly $[D]$. Moreover,
	for any line bundle $L$ with a global meromorphic section $s$, the transition
	functions $\frac{s_\alpha}{s_\beta}$ must just be the transition functions of
	$L$, and so $L = [(s)]$. From this we see that given a line bundle $L$, if there
	is a divisor $D$ such that $[D] = L$, we get a meromorphic section $s$ of $L$
	with $(s) = D$. Thus, a line bundle $L$ has a nonzero meromorphic section if and
	only if it is of the form $[D]$ for a divisor $D$. Moreover, it has a
	holomorphic section if and only if it is the line bundle of an effective
	divisor.  
	
	Let $\mathcal L(D)$ for a divisor $D$ denote the set of meromorphic functions
	$f$ such that $D + (f)$ is an effective divisor. Since $(f) = \sum_{V\subset M}
	\ord_V(f)V$, this is the set of all meromorphic functions which are holomorphic
	at any point not in a hypersurface in the divisor and where the order along any
	hypersurface is greater than the coefficient of that hypersurface in the divisor
	$D$. Recall that there is a global meromorphic section $s_0$ of $[D]$ with
	$(s_0) = D$; fix such an $s_0$. Then if $s$ is a global holomorphic section of
	$[D]$, the ratio $f = \frac s{s_0}$ is meromorphic and satisfies $(f) = (s) -
	(s_0) = (s) - D$. Equivalently, $D + (f) = (s)$; but $(s)$ is effective as the
	divisor associated to a holomorphic section, so $f\in \mathcal L(D)$.
	Conversely, if we pick a function $f\in\mathcal L(D)$, then $s = fs_0$ is a
	section of $[D]$ satisfying $(s) = (f) + (s_0) = (f) + D$, which is effective;
	thus $s$ must be a holomorphic section of $[D]$. Thus multiplication by $s_0$ identifies
	$\mathcal L(D)$ with $H^0(M, \mathcal O([D])) = \mathcal O([D])(M)$. 
	Because (by way of the function $(\bullet)$) a global section of a line bundle
	defines a divisor, we can extend this notion to a general line bundle:
	\begin{definition}
		Let $E$ be a linear subspace of $\mathcal
		O(L)(M)$ for some line bundle $L$ over a manifold $M$.  Then the
		\textbf{linear system} of divisors associated to $E$, which we call $|E|$,
		is the set $\{(s) | s\in E\}$.
	\end{definition}
	
	Also important is the notion of \textbf{linear equivalence}:
	
	\begin{definition}
		Let $D$ and $D$ be divisors on a manifold $M$. Then $D$ is \textbf{linearly
		equivalent} to $D$ if $D = D + (f)$ for a {\normalfont holomorphic} function
		$f$. Equivalently, if $[D] = [D]$. We denote linear equivalence by $\sim$.
	\end{definition}
	A short check yields that this is an equivalence relation.
	Recall that a linear subspace of projective space is the zero set of a
	homogeneous polynomial of the form $\sum_{i=0}^k a_ix_i$, and to each vector
	space $V$ we can associate a projective space $\mathbb P(V)$ which is the set of
	lines through the origin of $V$. We denote by $|D|$ the
	set of all effective divisors linearly equivalent to $D$.

\section{Blowing Up}
\begin{remark}We now treat a classical construction called the `Blowing Up' of a point,
	working first locally.
	Consider the subset $B$ of $\mathbb C^n\times \mathbb
	P^{n-1}$ satisfying the set of homogeneous polynomials \[x_iy_j - x_jy_i = 0\] (for points $((x_1, ..., x_n), [(y_1, ...,
	y_n)])$ in $\mathbb C^n\times \mathbb P^{n-1}$).  Define the map $\pi\colon B\to\mathbb
		C^n$ by projection on the first coordinate, Observe that for $\mathbb C-\{0\}$,
		the point $x$ lies on only one line through the origin; however, if $x=0$,
		then $x$ lies on every line through the origin. Thus $\pi$ is an isomorphism
		at every point other than the origin and $\pi^{-1}(0) = \mathbb P^{n-1}$
		is a divisor. 
		We call this construction the \textbf{blowing up of $\mathbb C^n$ at
		the origin}.
\end{remark}
\begin{remark}
		Note that these relations hold
		if and only if $x$ lies on the line that $y$ defines in $\mathbb P^{n-1}$,
		the set of all lines through the origin in $\mathbb C^{n}$. Furthermore,
		in general, this construction will yield a manifold with boundary
		given a manifold. Our theory holds for manifolds with boundary,
		although we don't make much of that explicit. 
\end{remark}
\begin{remark}\label{global}
	We can now perform the same construction in local coordinates on a manifold.
		Pick a point $p\in M$. 
	Restrict the above construction $B$ to a disc $D$ around the origin.
	Pick local coordinates $x_i$ about the point $p$ which are 0 at $p$, and perform the
	construction above for those local coordinates, identifying $D -
	\pi^{-1}(0)$ with the disk in $M$ surrounding $p$.
	This construction is independent of coordinates; if we perform it in both
		coordinates $z_i$ and $z'_i = fz_i$ yielding manifolds $\tilde M$ and
			$\tilde M'$, (letting $E = \pi^{-1}(x)$ and $E'=\pi'^{-1}(x)$, and
			$\tilde U$, $\tilde U'$ be open sets containing $E$ and $E'$ on
			which we have coordinates) there
			is a natural isomorphism induced by the coordinates $f: \tilde
			U-E \to \tilde U'-E$. We can extend this over $E$ by letting $f(x,
			l) = (x, l')$ where $l'_j = \pd[f_j]{z_i}(x)\cdot l_i$.
	\end{remark}
	\begin{definition} 
	We call this construction the \textbf{blowing up of $M$ at a point $x$}
	\end{definition}
	\begin{remark}
		Again, this yields a manifold with boundary. Moreover, we can repeat
		this for multiple disjoint points to yield the blowup at a finite number
		of points. 
	\end{remark}
	\begin{definition}
		If $\tilde M$ is the blowup of $M$ at $x$, we call $E_x = \pi^{-1}(x)$
		the \textbf{exceptional divisor} of $M$. 
	\end{definition}

	\begin{remark}
		Throughout the following, we'll be working a lot with restrictions onto
		points or closed sets. Although these generally aren't defined for a
		sheaf, they are for the ones we're working with in particular. 
	\end{remark}
	\begin{remark}\label{dude}
		Let $U$ be an open set of $M$ containing $x$ on which we can find local
		coordinates $z = (z_1, ..., z_n)$. Then we have 
		\[\tilde U = \pi^{-1}(U) = \{(z, l)\in U\times \mathbb P^{n-1} : z_il_j
		= z_jl_i\}\]
		Let $U_i = \{(z, l)\in \pi^{-1}(U): l_i\neq 0\}$; then we have an open
		cover of $\tilde U$. Moreover, in each $U_i$, we have local coordinates
		\[z_j^i\coloneq \begin{cases}\frac{z_j}{z_i} = \frac{l_j}{l_i} & j\neq
		i\\ z_i &j=i\end{cases}\].  
		These $z_i^j$ define precisely a global section of the quotient sheaf
		$\mathfrak M^*/\mathcal O^*$; they are locally meromorphic functions
		which differ globally by holomorphic multiples. Moreover, 
		the divisor $E$ is then given by $E = (z^i_i)$, so the line bundle $[E]$
		is given in intersections $U_i\cap U_j$ by transition functions
		$\frac{z_i^i}{z_j^j} = z^i_j = \frac{l_j}{l_i}$. As such, we can write
		the fibers of this line bundle in $U_i$ as $[E]_{(z, l)} =
		\{\frac{\lambda}{l_i} (l_1, ..., l_{i-1}, 1, l_{i+1}, ..., l_n) :
		\lambda\in\mathbb C\}$ (clearly this satisfies the transition functions). Equivalently, we can write the line bundle
		everywhere in $U_i$ as $[E]_{(z, l)} = \{{\lambda}(l_1, ..., l_{i-1}, l_i, l_{i+1}, ..., l_n) :
		\lambda\in\mathbb C\}$. Restricted to $E\cong \mathbb P^{n-1}$, this is
		the \textbf{tautological bundle}, the bundle who's fiber over each point
		in $\mathbb P^n$ is the line which that point represents. The dual of
		the tautological bundle is the \textbf{hyperplane bundle}, for which the fiber
		over a line $l\in\mathbb P^n$ is the space of linear functionals from $l$ to
		$\mathbb C$; as such, $[-E] = [E]^*$ is the hyperplane bundle. Then,
		given a function $f$ which vanishes at $x$, the pullback $\pi^*f$ is a
		function which vanishes along all of $E$; that is, locally a section of
		$[-E]$.

		Picking a set of coordinates for $M$ gives us an identification of $E$
		with the projectivization of the tangent space by the map 
		\[(x, l) \mapsto \sum l_i\pd{x_i}\]
		This extends between coordinates; if $x', l'$ are new coordinates we
		have  
		\[(x, l) = (f_1x_1, ..., f_nx_n, g_1l_1, ..., g_nl_n) \mapsto
		\sum_i\sum_j \pd[f_j]{x_i}l_j\pd{x'_j}\sim \sum l_i\pd{x_i}\]
		By the identification on the tangent space. From this we see that there
		is a very natural identification between the projectivization of the
		tangent space and the exceptional divisor. This identification extends
		naturally to the cotangent space; $[-E]$ is the set of fiberwise linear
		functionals from $[E]$ to $\mathbb C$, and when restricted to $E$, is
		thus the set of linear functionals from the tangent space to $\mathbb
		C$. But recall that a section of $[-E]$ is also represented as a
		collection of complex-valued functions satisfying transition functions;
		in particular, picking coordinates for $M$ around $x$ induces
		coordinates for $\tilde M$, including coordinates $\pd{x_i}$ for $E$.
		The pullback of $f$ induces new coordinates $f(x_1),
		..., f(x_n)$ and obeys the transition rules; as such, we must have by
		change of coordinates
		\[\pi^*f|_E = \sum_{i=1}^{n}\frac{\partial f}{\partial x_i}\Big|_{x} dx_i\]
		But this is exactly the differential $df$ evaluated at $x$; from this we obtain the
		commutative diagram in Figure \ref{fig:dig},
	\begin{figure}
		\centerline{
	\xymatrix{
		\check H^0(\tilde M, \mathcal O({[-E]}))
			\ar[r]^{r'_E}&\check H^0(E, \mathcal O([-E]))\\
			\check H^0(M, \mathcal O({\mathscr I_x}))\ar[u]\ar[r]^{d_x}&
			T^*_xM\ar@{=}[u] }}
		\caption{A commutative diagram.}
		\label{fig:dig}
	\end{figure}
	where $d_{x}$ is the restriction of the derivative onto the point $x$, $r'_E$ the
	restriction map $r'_E(f) = \sum n_i f|_{V_i}$ onto the divisor $E = \sum
		n_iV_i$, and $\tilde L$ is the pullback $\pi^*L$. 
	\end{remark}
\begin{lemma}\label{pullback}
		Let $\tilde M$ be the blow up of $M$ at $x$ with projection $\pi$.
		If $K_M$ is the canonical bundle on an $n$-manifold $M$, $\tilde K_M$ the pullback
		$\pi^*(K_M)$, and $K_{\tilde M}$ the canonical bundle on $\tilde M$,
		then $K_{\tilde M} = \tilde K_M\otimes [(n-1)E]$
\end{lemma}
	\begin{proof}
		We begin by constructing an open cover
		$\{U_x\}\cup\{U_\alpha\}_{\alpha\in A}$, where $x\in U_x$ but $x\not\in
		U_\alpha$ for any $\alpha$, and where the intersections $U_x\cap
		U_\alpha$ lie in a single chart of $M$ (this can always be
		accomplished by shrinking $U_x$ and growing the $U_\alpha$). Let
		$\tilde U_\alpha$ be the pullback of $U_\alpha$ under $\pi$, and let
		$\tilde U_i$ be $U_x\cap \{(x, l) | l_i\neq 0\}$. Let $z_j^i$ be
		coordinates in the second set of open covers, and $w^\alpha_i$
		coordinates in the first. Recall from the universal property of the
		determinant that the transition functions of the canonical bundle may be
		computed as the determinant of the Jacobian of the transition functions
		of the manifold; as such, we seek to compute the transition functions of
		the manifold on this open cover. 

		Firstly recall from Remark \ref{dude} that in $U_j\cap U_k$, we
		have\footnote{We work here, somewhat implicitly, in \textit{two} sets of
		local coordinates, writing the first (the $z_i^j$) in terms of the
		second (the $z_i$, which are coordinates for $M$ rather than $\tilde M$).}
		$z^j_i = \frac{z_i}{z_j}$ for $i\neq j$ and $z^j_j=z_j$. Likewise,
		$z^k_i = \frac{z_i}{z_k}$ for $i\neq k$ and $z^k_k=z_k$. From this we
		can write transition functions in terms of the coordinates on $M$; if we
		write $f_j = (z_1^j, ..., z_n^j)$, then $f^{-1}_j(z_1, ..., z_n) =
		(z_1z_j, ..., z_j, ..., z_nz_j)$, and so  is 
		\[f_k\circ f^{-1}_j(z_1, ..., z_n)=\left(\frac{z_1z_j}{z_k}, ...,
		\frac{z_j}{z_k},...,z_kz_j, ..., \frac{z_nz_j}{z_k}\right)\]
		The $l$-th partial derivative of this is 
		\[\pd[f_k\circ f^{-1}_j]{z_l}=\begin{cases}
			\left(0, ...,
			\frac{z_j}{z_k},...,0\right)\text{  with the $l$-th coordinate
			nonzero,}& l\neq
			j, l\neq k\\
	\left(0, ..., \frac{1}{z_k},...,0\right) \text{  with the $l$-th coordinate
			nonzero,}& l=j\\
			\left(-\frac{z_1z_j}{(z_k)^2}, ..., -\frac{z_j}{(z_k)^2},...,z_j, ...,
			-\frac{z_nz_j}{(z_k)^2}\right) & l=k\\
		\end{cases} \]
		The determinant of this Jacobian is then:
		\[\left(\prod_{l=0, l\neq j, l\neq k}^{n}
		(-1)^{l}\frac{z_j}{z_k}\right)\left(\frac{(-1)^j}{z_k}\right)(-1)^k\left(z_j\right)\]
		\[=(-1)^{j-k+n}\left(\frac{z_j}{z_k}\right)^{n-2}\left(\frac{(-1)^j}{z_k}\right)(-1)^k\left(z_j\right)\]
		\[=(-1)^{n}\left(\frac{z_j}{z_k}\right)^{n-1}\]
		Similarly, in $U_i\cap U_\alpha$ we can write $w_j^\alpha =
		z_iz_j^i = z_j$, yielding functions $f_{i} = (z_1^i, ..., z_n^i)$ and
		$f_\alpha = (w_1^\alpha, ..., w_n^\alpha) = (z_1, ...,  z_n)$. Then
		\[f_i\circ f_\alpha^{-1} = \left(\frac{z_1}{z_i}, ...,z_i,...,
		\frac{z_n}{z_i}\right)\]
		Taking derivatives yields:
		\[\pd[f_i\circ f_\alpha^{-1}]{z_j} =\begin{cases}
		\left(0, ..., \frac{1}{z_i}, ..., 0 \right) & i\neq j\\
		\left(0, ...,1,..., 0\right) & i=j
		\end{cases} \]
		Meaning the Jacobian is
		\[(-1)^j\prod_{i\neq j}^n\frac{(-1)^i}{z_i} =
		\frac{(-1)^{n}}{(z_i)^{n-1}}\]
		Everywhere else the transition functions are the pullback of those on
		$M$. To compile the data we have so far, the transition functions
		$g_{\bullet\bullet}$ for
		the canonical line bundle on $\tilde M$ can be written in terms of
		coordinates $z_i$ on $M$ as:
		\[\begin{cases}
			g_{kj} = (-1)^n\left(\frac{z_j}{z_k}\right)^{n-1}\\
			g_{i\alpha} = \frac{(-1)^n}{(z_i)^{n-1}}\\
			g_{\alpha\beta} = \pi^*g'_{\alpha\beta}
		\end{cases}\]
		For $g'_{\alpha\beta}$ the transition functions of $K_M$ on $M$. Note
		that multiplication by $-1$ is a holomorphic function; as such, we can
		ignore it and retain the same line bundle.
	We already know that $[E]$ is given by transition functions $h_{kj}\frac{z_k}{z_j}$ in $U_j\cap U_k$;
	moreover, it is given by transition functions $h_{\alpha i}=z_i$ in $U_i\cap
	U_\alpha$ and as $h_{\alpha\beta}=1$ everywhere else. But then we have that
	the line bundle $K_{\tilde M}\otimes [(n-1)E] = K_{\tilde M}\otimes
	[E]^{\otimes (n-1)}$ is given by transition functions
		\[\begin{cases}
			e_{kj} = \left(\frac{z_j}{z_k}\right)^{n-1} \cdot
			\left(\frac{z_k}{z_j}\right)^{(n-1)} =1\\
			e_{i\alpha} = \frac{1}{(z_i)^{n-1}}\cdot (z_i)^{n-1}=1\\
			e_{\alpha\beta} = \pi^*g'_{\alpha\beta}\cdot 1 = \pi^*g'_{\alpha\beta}
		\end{cases}\]
		Which is exactly the pullback of the line bundle given by
		$g'_{\alpha\beta}$, and the lemma is proven. 
\end{proof}
\begin{remark}
		There is a natural Kähler metric on projective space $\mathbb P^n$ called the
		\textbf{Fubini - Study metric}. In local coordinates on the open sets
		$U_i =  \{[x_1: ... : x_n] | x_i \neq 0\}$ with maps $[x_1, ... ,
		x_n]\mapsto \left(\frac{x_1}{x_i},\dots, \hat z_i, \dots, \frac{x_n}{x_i}\right)$ we have 
		\[\omega_i\coloneq \frac
		i{2\pi}\partial\bar\partial\log\left(\sum_{k=1}^n\left|\frac{z_l}{z_i}\right|^2\right)\]
		This glues togehter to a global form as
		\[\log\left(\sum_{l=1}^n\left|\frac{z_l}{z_i}\right|^2\right) = 
		\log\left(\sum_{l=1}^n\left|\frac{z_j}{z_i}\right|^2\left|\frac{z_l}{z_j}\right|^2\right) = 
		\log\left(\left|\frac{z_j}{z_i}\right|^2\right)+\log\left(\sum_{l=1}^n\left|\frac{z_l}{z_j}\right|^2\right)\]
		 But the left term vanishes under $\partial\bar\partial$ by a result
		 from complex analysis, so the expression is well defined between charts
		 and we can glue together to get a function on the whole of $\mathbb
		 P^n$. Note that $\overline{\partial\bar\partial} = \bar\partial\partial
		 = - \partial\bar\partial$, together with $\bar i = -i$, yields that
		 $\bar\omega = \omega$. Finally, a straightforward computation in coordinates shows that
		 $\omega$ is positive-definite. 
\end{remark}
\begin{lemma}\label{pos}
		Let $\tilde M$ be the blow up of $M$ at $h$ points, $n_1, ..., n_h$ be
		positive integers, $E_i$ the
		exceptional divisors of these points and $\pi$ be the
		associated projection map. For any positive line bundle $L$, and any
		line bundle $K$, there
		exists an integer $k\gg 0$ such that $\pi^*(L^{\otimes k}\otimes K)\otimes
		\left[\sum -n_i E_i\right]$ is positive.
\end{lemma}
	\begin{proof}
		The Fubini-Study metric induces a metric on the hyperplane bundle on
		$\mathbb P^{n-1}$; we pull this back\footnote{Although this pullback does depend on the choice of
		coordinates, the following discussion holds true for all of them; as
		such, we can simply pick our favorite and work in that without loss of
		generality} to $[-E_i]$ by Remark
		\ref{dude}. We
		can glue the $n_i$-th power of these metrics together by means of a
		partition of unity to obtain a form on 
		\[\left[\sum_{i}-n_iE_i\right]\]
		Where we let the function far away from each of the $E_i$ be zero. Then the curvature form
		$\omega = -n_ip^*\omega_{\text{FS}}$ (where $p$ the biholomorphism to projective space)
		satisfies $in_i\omega(v, \bar v)(x) \geq 0$ everywhere; the
		curvature form is \textbf{semi-positive}. But then if we let $\alpha$
		and $\beta$ be the curvature forms given by $L$ and $K$ respectively.
		Since $\alpha$ is positive, linearity gives that for large enough $k$ we
		have $-i(k\alpha + \beta)(v, \bar v) = -ik(\alpha)(v, \bar v) 
		-i\beta(v, \bar v)>0$. Pulling back these forms along the holomorphic
		immersion $\pi$ preserves positivity, so
		linearity again yields that 
		\[-i(\pi^*(k\alpha + \beta) + \omega)(v, \bar v) = -ik(\pi^*\alpha)(v, \bar v) 
		-i\pi^*\beta(v, \bar v)  -i\omega(v, \bar v) >0\]
		But the form $(\pi^*(k\alpha + \beta) +\omega)$ is, by Lemma \ref{add},
		the curvature form on \[\pi^*(L^{\otimes k}\otimes K)\otimes
		\left[\sum -n_i E_i\right],\] and we are done.
	\end{proof}
\section{The Kodaira Embedding Theorem}
Our goal in this section is to prove the main result of the paper:
\begin{theorem}[Kodaira Embedding Theorem]
	Let $M$ be a compact, complex manifold, and $L$ a positive line bundle on
	$M$. Then $M$ is projective. In particular, there is a $k\gg0$ such that
	$L^{\otimes k}$ determines an embedding $\iota$ of $M$ into projective space.
\end{theorem}
Our primary tool for proving the theorem will be 
\begin{lemma}\label{main}
	A holomorphic line bundle $L$ on $M$ with no base points (no points where every
	section vanishes) and a choice of basis for $\mathcal O({L},M)$ taken
	together yield a smooth map $\iota:M\to \mathbb P^N$, where $N = \dim(
	\mathcal O(L ,M)) - 1$\footnote{Recall that $\mathcal O(L, M)\cong\check H^0(M,
	\mathcal O(L))\cong H_{dR}^0(M, \mathcal O(L))\cong H^{(0, 0)}(M, \mathcal
	O_L))\cong \mathcal H^{0,0}(M, \mathcal O(L))$ is finite dimensional.}.
\end{lemma}
\begin{proof}
	Let $L$ be a holomorphic line bundle and $\{s_i\}_{i=0}^N$ a basis for the
	$N+1$-dimensional space of sections of $L$. Then $\iota_L:p\mapsto [s_0,
	..., s_N]$ is a well-defined holomorphic map.
\end{proof}
\begin{lemma}\label{ses}
	If $\tilde M$ is $M$ blown up at two points $x$ and $y$, $E_x$ and $E_y$ are the
	exceptional divisors, and $E = E_x + E_y$, we have the following short exact
	sequence:
	\[0\to \mathcal O({L\otimes [-E]})\to \mathcal
	O({L})\xto {r_E}\mathcal O(L)|_E\to 0\]
	For any line bundle $L$ on $\tilde M$, 
	where $\mathcal O({L})|_E$ is the set of all sections of $L$
	on the single hypersurface of the divisor $E$ and $r'_E$ is the restriction to the divisor.
\end{lemma}
\begin{proof}
	The restriction map is clearly surjective. If we suppose $\{U_{\alpha}\}$ is
	a trivialization for $L$ with transition functions $g_{\alpha\beta}$, and
	$\{U_{i}, f_{ij}\}$ the trivialization and transition functions described in
	Remark \ref{dude}, we have that $L\otimes [-E]$ is given by transition
	functions $\frac{g_{\alpha\beta}}{f_{ij}}$ where $(f_{ij}) = [E]$. We can
	then write a section of $L\otimes [-E]$ as a collection of functions
	$\sigma_{\alpha i \beta j}: U_{\alpha}\cap U_i\cap U_\beta\cap U_j\to
	\mathbb C$ satisfying the appropriate transition functions. As such,
	we consider the function $\phi:\mathcal O(L\otimes [-E])\to \mathcal
	O(L)$
	taking $\sigma = \{\sigma_{\alpha i \beta j}\} \mapsto \{f_{ij}\sigma_{\alpha i
	\beta j}\}$. Clearly this is injective; however, it's image will be the set
	of sections of $L$ which are zero on $E$. But this is exactly the
	kernel of $r_E$, and we are done.
\end{proof}
	\begin{corollary}
	If $\tilde M$ is the blowup of $M$ at a point $x$ and $E$ the exceptional
		divisor, we have a short exact
	sequence for any line bundle $L$ on $M$:
		\[0\to \mathcal O({\tilde L^{\otimes k}\otimes [-2E]}) \to \mathcal O({\tilde
		L^{\otimes k}\otimes [-E]})\xto{r'_E}\mathcal O({\tilde L^{\otimes k}\otimes
		[-E]})|_E\to 0\]
	Where $\mathcal O({\tilde L^{\otimes k}\otimes
		[-E]})|_E$ is the set of all sections of $\tilde L^{\otimes k}\otimes [-E]$
	on the two hypersurfaces of the divisor $E$, $\tilde L = \pi^*L$ is the pullback of $L$, and $r'_E$ is
		as the restriction to the divisor.
\end{corollary}
	\begin{proof}
		Follows from Lemma \ref{ses} by letting $ L = \tilde L^{\otimes
		k}\otimes [-E]$
	\end{proof}
\begin{theorem}[Kodaira Embedding Theorem]
	Let $M$ be a compact, Kähler manifold, and $L$ a positive line bundle on
	$K$. Then $M$ is projective. In particular, there is a $k\gg0$ such that
	$L^{\otimes k}$ determines an embedding $\iota$ of $M$ into projective space.
\end{theorem}
	\begin{proof}
		As we have a smooth map $\iota_L$ by Lemma \ref{main}, it suffices to show that the map separates
		points and tangents; that is, that the map is an injective immersion.
		We will first show injectivity. If injectivity fails, there exist points
		$x, y$ such that $\iota_L(x) = \iota_L(y)$. But by the definition of
		$\iota_L$, this means that there exists a scalar $\lambda$ such that for
		any section $s$, we have $\lambda s(x) = s(y)$. In particular, the image
		of the restriction / evaluation map $r_{x, y}:s\mapsto (s(x), s(y))\in
		L^{\otimes k}_{x}\oplus L_y^{\otimes k}$ is the set of all $(s(x),
		\lambda s(x))$ for $x\in M$; as $\lambda$ is fixed, this is not the
		whole set. By this argument, it suffices to show that $r_{x, y}$ is
		surjective. Let $\tilde M$ be the blowup of $M$ at two distinct points $x$ and $y$, with
	exceptional divisors $E_x$ and $E_y$. Then by definition and Hartogs'
		theorem we have the
		commutative diagram:

		\begin{figure}[h]
	\centerline{
	\xymatrix{
		\check H^0(\tilde M, \mathcal O({\tilde L^{\otimes k}}))
			\ar[r]^{r_E}&\check H^0(\mathcal O({\tilde L^{\otimes
		K}})|_E)\\
			\check H^0(M, \mathcal O({L^{\otimes k}}))\ar[u]\ar[r]^{r_{x, y}}&
			L_x^{\otimes k}\oplus L_y^{\otimes k}\ar@{=}[u]
	}}
	\end{figure}
	Where $r_E$ is the
	restriction map $r_E(f) = \sum n_i f|_{V_i}$ onto the divisor $E = E_x +
	E_y = \sum n_iV_i$, and $\tilde L$ is the pullback $\pi^*L$. By the diagram,
		it suffices to show that $r_E$ is surjective. But by Lemma \ref{ses} we have the short
		exact sequence
	\[0\to \mathcal O({\tilde L^{\otimes k}\otimes [-E]})\to \mathcal
	O({L^{\otimes K}})\xto {r_E}\mathcal O({\tilde L^{\otimes K}})|_E\to 0\]
	Which induces a long exact sequence on the level of cohomology:
		\[0\to \check H^0(\mathcal O({\tilde L^{\otimes k}\otimes [-E]}))\to
		\check H^0(\mathcal
		O({L^{\otimes K}}))\xto {r_E}\check H^0(\mathcal O({\tilde L^{\otimes
		K}})|_E)\]\[\to
		\check H^1(\mathcal O({\tilde L^{\otimes k}\otimes [-E]}))\to \dots\]
	So it suffices to show that 
		$\check H^1(\mathcal O({\tilde L^{\otimes k}\otimes [-E]})) =0$ for some
		$k$. 
		Let $\tilde K_M = \pi^*K_M$. By Lemma \ref{pullback} we have that $K_{\tilde M} = \tilde
	K_M\otimes[(n-1)E]$ (where $n = \dim_\mathbb C(M)$). Then 
	\begin{align*}\tilde L^{\otimes k} \otimes [-E] =&
		K_{\tilde M}\otimes \tilde L^{\otimes k}\otimes [-E]\otimes K^*_{\tilde
		M}\\
		=&K_{\tilde M}\otimes (L^{\otimes (k-k')}\otimes\tilde K^*_M)\otimes (\tilde
		L^{\otimes k'}\otimes [-nE])
	\end{align*}
	But by Lemma \ref{pos}, we can choose $k$ and $k'$ so that $(L^{\otimes (k-k')}\otimes\tilde K^*_M)\otimes (\tilde
		L^{\otimes k'}\otimes [-nE])$ is a positive line bundle, meaning we can apply
		the Kodaira Vanishing Theorem to show that the first cohomology
		vanishes. 

	 We now seek to show the differential of $\iota_E$ is injective everywhere;
		i.e. that $\iota_E$ separates tangents. 
			 By the rank-nullity theorem, it suffices to show that $d\iota$
				has the same rank as $T^*M$ at $x$ 
			 Let $\mathscr I_x$ denote the set of sections of $L$ vanishing
				at $x$, and 
			 suppose $s\in \mathscr I_x(M)$. If we have distinct
				trivializations $\psi_\alpha$ and $\psi_\beta$ for the line
				bundle $L$ about the point $x$, then we note
				that $\psi_\alpha^*s = g_{\alpha\beta} \psi_\beta^*s$ (here the
				star denotes the ``pushforward'' of a section of $\mathscr I_x$ on
				$M$ to a $\mathbb C$-valued function on $M$ by way of the
				trivialization $\psi$). Thus by
				the chain rule and the fact that $s$ vanishes at $x$:
				\begin{align*}
					\psi\circ d(\psi_\alpha^*s) = & \psi\circ d(g_{\alpha\beta}\psi_\beta^*s)
					=  d(g_{\alpha\beta})\psi_\beta^*s +
					d\psi_\beta^*sg_{\alpha\beta}
				= g_{\alpha\beta} d\psi_\beta^*s 
				\end{align*}
			 So we may construct a well defined map  
			 \[d_x:H^0(M, \mathscr
				I_L)\to T^*_{\partial, x}\otimes L_x\text{   by the rule   }
				d_x: s \mapsto \left[d(\psi_\alpha^*
				s)\right]_{x}\]. 
			 If this map is surjective, then for each $v\in T^*_xM$ we can
				find an $s$ with 
				\[ds = \sum_{i=1}^n\frac{\partial s}{\partial
				x_n}dx_n = v,\] where the $x_n$ are holomorphic coordinates. Thus
				the differential $d\iota$ has the same rank as $T^*M$. 

		To show this, let $\tilde M$ denote the
		blow-up of $M$ at $x$; let $E$ denote the exceptional divisor, and $\pi$
		the projection. The pullback map 
		\[\pi^*: \check H^0(M, \mathcal O(L^{\otimes k}))\to \check H^0(\tilde
		M, \mathcal O_{\tilde M}(\tilde L^{\otimes k}))\]
		is then an isomorphism. Moreover, this restricts to an isomorphism 
		\[\pi^*: \check H^0(M, \mathcal O(\mathscr I_x(L))\to \check H^0(\tilde
		M, \mathcal O_{\tilde M}(\tilde L^{\otimes k}\otimes [-E]))\]
		Because a section of $L$ vanishes on $x$ if and only if the pullback of
		that section vanishes along all of $E$; as we have discussed, sections
		which vanish along $E$ are sections of $[-E]$. But by identifying $\check
		H^0(\mathcal O(\tilde L^k-E)|_E)$ with $L_x^{\otimes k}\otimes
		T^*_{\partial, x}$ and referencing the diagram in Remark \ref{dude}
		(as tensoring preserves the commutivity), we obtain that the following
		diagram commutes:

	\begin{figure}[h]
		\centerline{
	\xymatrix{
		\check H^0(\tilde M, \mathcal O({\tilde L^{\otimes k}\otimes [-E]}))
			\ar[r]^{r'_E}&\check H^0(E, \mathcal O({\tilde L^{\otimes
			k}\otimes[-E]}))\\
			\check H^0(M, \mathcal O({\mathscr I^k_x}))\ar[u]\ar[r]^{d_x}&
			T^*_xM\otimes L_x^{\otimes k}\ar@{=}[u] }}
	\end{figure}

	On $\tilde M$ by the Corollary to Lemma \ref{ses} we have the exact sequence 
	\[0\to \mathcal O({\tilde L^{\otimes k}\otimes [-2E]}) \to O({\tilde
	L^{\otimes k}\otimes [-E]})\xto{r_E}\mathcal O({\tilde L^{\otimes k}\otimes
	[-E]})|_E\to 0\]
	By the snake lemma it suffices to show $H^1(M, \mathcal O({\tilde L^{\otimes k}\otimes
	[-E]})) = 0$.
	We can choose $k$ and $k'$ such that $L^{\otimes(k-k')}\otimes K^*_M$ is
	positive on $M$ and $\tilde K^{\otimes k'}\otimes[(-n-1)E]$ is positive on
	$\tilde M$. Pulling back the first yields that $\tilde L^{\otimes
	(k-k')}\otimes \tilde K^*_M$ is positive on $\tilde M$. But then we can write 
	\begin{align*}\tilde L^{\otimes k}\otimes [-2E] =& K_{\tilde M}\otimes (\tilde
	L^{\otimes(k-k')}\otimes K^*_{\tilde M}\otimes \tilde L^{\otimes k'})\\
	=& K_{\tilde M}\otimes (\tilde
	L^{\otimes(k-k')}\otimes \tilde K^*_{M}\otimes \tilde L^{\otimes k'}\otimes
		[(-n-1)E])
	\end{align*}
	With $(\tilde
	L^{\otimes(k-k')}\otimes \tilde K^*_{M}\otimes \tilde L^{\otimes k'}\otimes
		[(-n-1)E])$ positive, and so we apply the Koidara Vanishing Theorem to
		show the first cohomology vanishes. 

	Picking the maximum of the $k$s found in both of these steps, we have shown
	that at each point there exists a $k$ such that the map
	associated to $L^{\otimes k}$ separates both points and tangents. But as the
	set on which any given function is an embedding is open, there is a
	neighborhood $U_x$ around each point $x$ for which $\iota_{L^{\otimes k_x}}$
	is an embedding. This forms an open cover of $M$, as each point $x$ is in
	$U_x$. But as $M$ is compact, we can find a finite sub-cover of this open cover,
	and with it, a finite number of $k$s which suffice to ensure $\iota$ is
	positive. Picking the maximum of these $k$s then yields a single $k$ for
	which $\iota_{L^{\otimes k}}$ is an embedding.
	\end{proof}
%
\nocite{*}
\printbibliography

\end{document}